\documentclass[12pt,reqno]{amsart}

\newtheorem{lemma}{Lemma}[section]
\newtheorem{definition}{Definition}[section]
\newtheorem{proposition}[lemma]{Proposition}
\newtheorem{theorem}[lemma]{Theorem}
\newtheorem{remark}{Remark}[section]

\newtheorem{question}{Question}[section]

\newtheorem{corollary}[lemma]{Corollary}
\newtheorem{example}{Example}[section]

\usepackage{amssymb,amsmath,amsthm}
\usepackage{epsfig,graphicx,psfrag}
\usepackage[all]{xy}
\usepackage{enumerate}
\usepackage{animate} 
\usepackage{color}

\begin{document}

\title{Flat affine manifolds and their transformations }


\author{Medina A.$^1$,  Saldarriaga, O.$^2$, and Villabon, A.$^2$}

\subjclass[2020]{Primary: 57S20, 54H15; Secondary: 53C07, 17D25  \\ Partially Supported by CODI, Universidad de Antioquia. Project Number 2015-7654.}
\date{\today}

\begin{abstract} 
We give a characterization of flat affine connections  on  manifolds by means of a natural affine representation  of the universal covering of the Lie group of diffeomorphisms preserving the connection. From the infinitesimal point of view, this representation is determined by the  1-connection form and  the fundamental form of the bundle of linear frames of the manifold. We show that the group of affine transformations of a real flat affine $n$-dimensional manifold, acts on $\mathbb{R}^n$ leaving an open orbit when its dimension is greater than $n$. Moreover, when the dimension of the group of affine transformations is $n$, this orbit has discrete isotropy.  For any given  Lie subgroup $H$ of affine transformations of the manifold,  we show the existence of an  associative envelope of the Lie algebra of $H$, relative to the connection. The case when $M$ is a Lie group and $H$ acts on $G$ by left translations is particularly interesting.  We also exhibit some results about  flat affine manifolds whose group of affine transformations admits a flat affine bi-invariant structure.  The paper is illustrated with several examples.
\end{abstract}
\maketitle


\vskip5pt
\noindent
 $^1$ Institute A. Grothendieck,   Universit\'e  Montpellier, UMR 5149 du CNRS, France and Universidad de Antioquia, Colombia

e-mail: alberto.medina@umontpellier.fr  
\vskip5pt
\noindent
 $^2$ Instituto de Matem\'aticas, Universidad de Antioquia, Medell\'in-Colombia

 e-mails: omar.saldarriaga@udea.edu.co, andresvillabon2000@gmail.com

\section{Introduction}
 
The objects of study of this paper are flat affine paracompact smooth manifolds with no boundary and their affine  transformations. A well understanding of the category of Lagrangian manifolds assumes a good knowledge of the category of flat affine manifolds (Theorem 7.8 in \cite{W}, see also \cite{F}). Recall that flat affine manifolds with holonomy reduced to $GL_n(\mathbb{Z})$ appear naturally in integrable systems and Mirror Symmetry (see \cite{KS}). 

For simplicity, in what follows $M$ is a connected real n-dimensional manifold, $P=L(M)$ its bundle of linear frames, $\theta$ its fundamental 1-form, $\Gamma$  a linear connection on $P$ of connection form  $\omega$  and  $\nabla$ the covariant derivative on $M$ associated to $\Gamma$. The pair $(M,\nabla)$ is called a flat affine manifold if the curvature and torsion tensors of $\nabla$ are both null. In the case where $M$ is a Lie group and the connection is left invariant, we call it a flat affine Lie group.  It is well known that a  Lie group is flat affine if and only if its Lie algebra is endowed with a left symmetric product compatible with the bracket. For a rather complete overview of flat affine manifolds, the reader can refer to \cite{Gol}. 

 An affine transformation of $(M,\nabla)$ is a diffeomorphism $f$ of $M$  whose derivative map $f_*:TM\longrightarrow TM$ sends parallel vector fields into parallel vector fields, and therefore geodesics into geodesics (together with its affine parameter). An infinitesimal affine transformation of $(M,\nabla)$ is a smooth vector field $X$ on $M$ whose local 1-parameter groups $\phi_t$ are local affine transformations. We will denote  by $\mathfrak{a}(M,\nabla)$ the real vector space of infinitesimal affine transformations of $(M,\nabla)$ and for $\mathfrak{X}(M)$ the Lie algebra of smooth vector fields on $M$. An element $X$ of $\mathfrak{X}(M)$ belongs to $ \mathfrak{a}(M,\nabla)$ if and only if it satisfies
\begin{equation}\label{Eq:ecuacionkobayashi} \mathcal{L}_X\circ \nabla_Y -\nabla_Y\circ \mathcal{L}_X=\nabla_{[X,Y]},\quad\text{for all}\quad Y\in \mathfrak{X}(M),\end{equation}
where $\mathcal{L}_X$ is the Lie derivative. For a vector field $X\in\mathfrak{X}(M)$, we will denote  by $L(X)$ its natural lift on $P$. If $\phi_t$ is the flow of $X$, the flow  of $L(X)$ is the natural lift $L(\phi_t)$ of $\phi_t$ given by $L(\phi_t)(u)=(\phi_t(m),(\phi_t)_{*,m}X_{ij}),$ where $u=(m,X_{ij})$. It is also well known that
\begin{equation} \label{Eq:condition1oninfinitesimalaffinetransformations} X\in\mathfrak{a}(M,\nabla)\text{ if and only if }\mathcal{L}_{L(X)}\omega=0. \end{equation} This is also equivalent to saying that $L(X)$ commutes with every standard horizontal vector field $B(\xi)$ (i.e., vector fields satisfying that $\theta(B(\xi)):=\xi$  for all $\xi\in\mathbb{R}^n$). We will call $\mathfrak{a}(P,\omega)$ the set of vector fields $Z$ on $P$ satisfying 
\begin{align} \notag
& Z \text{ is invariant by }R_a\text{ for every }a\in GL_n(\mathbb{R})\\ \label{Eq:infinitesimalaffinetransf}
& \mathcal{L}_Z\theta=0 \\ \notag
&\mathcal{L}_Z\omega=0,
\end{align} 
where $R_a$ is the right action of $GL_n(\mathbb{R})$ on $P$. The map $X\mapsto L(X)$ is an isomorphism of Lie algebras from $\mathfrak{a}(M,\nabla)$ to $\mathfrak{a}(P,\omega)$. The vector subspace $\mathsf{aff}(M,\nabla)$ of $\mathfrak{a}(M,\nabla)$ whose elements are complete, with the usual bracket of vector fields,  is the Lie algebra of the group $\mathsf{Aff}(M,\nabla)$ of affine transformations of $(M,\nabla)$ (see \cite{KN} chapter VI, page 229). The image of $\mathsf{aff}(M,\nabla)$ under the isomorphism $X\mapsto L(X)$ will be denoted  by $\mathfrak{a}_c(P)$. For the purposes of this work, it is useful to recall that $(P,\omega)$ admits an absolute parallelism, that is, its tangent bundle admits $n^2+n$ sections independent at every point. In fact, $\{B(e_1),\dots,B(e_n),E_{11}^*,\dots,E_{nn}^*\}$ is a paralelism of $P$, where $(e_1,\dots,e_n)$ and $\{E_{11},\dots,E_{nn}\}$ are respectively the natural basis of $\mathbb{R}^n$ and $gl_n(\mathbb{R})$ (see \cite{KN} p 122).

\begin{remark} If $X\in\mathfrak{a}(M,\nabla)$ has flow $\phi_t$, the flow  $L(\phi_t)$ of $Z=L(X)$   determines a local isomorphism of $(P,\omega).$ This implies the existence of an open covering $(U_\alpha)_{\alpha\in\mathcal{A}}$ trivializing the bundle $(L(M),M,\pi)$ such that the maps $L(\phi_t):(\pi^{-1}(U_\alpha),\omega)\longrightarrow (\pi^{-1}(\phi_t(U_\alpha)),\omega)$ are isomorphisms. When $Z$ is complete, the maps $L(\phi_t)$ are global automorphisms of $(P,\omega)$. 
	
Notice that, for every $t$, the map $L(\phi_t)$ preserves the fibration and the horizontal distribution associated to $\omega$.\end{remark}

This paper is organized as follows. In Section \ref{S:newresults} we state a characterization of flat affine manifolds and provide some of its consequences.  In the theory of flat affine manifolds appears, in a natural way, a finite dimensional associative  algebra (see Lemma \ref{L:associativity}). This will allow to show the existence of an   associative envelope of the Lie algebra of the Lie group relative to a subgroup $H$ of the group of affine transformations of the manifold. In Section \ref{S:associativeenvelope} we specialize our study to the case of flat affine Lie groups and  exhibit some examples. Section \ref{S:studyofaffinetransformationgroups} deals with the study of the following question posed by A. Medina. 

\begin{question} \label{Q:Medinasquestion} Does the group of affine transformations, $\mathsf{Aff}(M,\nabla)$, of a flat affine manifold $(M,\nabla)$ admit a left invariant flat affine or, by default, projective  structure determined by $\nabla$?\end{question}
 This question was answered positively in some particular cases  in the special case of some flat affine Lie groups (see in \cite{MSG}). In Section 4. we exhibit more cases where the answer to this question is positive. 

\section{Characterization of flat affine  manifolds} \label{S:newresults}

Before stating the main result of the section, let us denote  by  $A^*$    the fundamental vector field on $P=L(M)$ associated to $A\in\mathfrak{g}=\mathsf{gl}_n(\mathbb{R})$ and by $B(\xi)$ the basic or standard horizontal vector fields (see \cite{KN} page 63). Recall that the fundamental form  $\theta$ of $P$ is a tensorial 1-form of type $(GL(\mathbb{R}^n),\mathbb{R}^n)$, that is, it verifies $(R_{a}^*\theta)(Z)=a^{-1}(\theta(Z))$, \ for all \ $a\in\mathsf{GL}_n(\mathbb{R})$ and $Z\in\mathfrak{X}(P)$. 

If  $\Gamma$ is a linear connection over $M$, its connection  form $\omega$ is a  $gl_n(\mathbb{R})$-valued 1-form satisfying  
\begin{align}\label{Eq:omegafixesfundamentals}
&\omega(A^*)=A,\quad\text{for all} \quad A\in\mathsf{gl}_n(\mathbb{R}),\\ \label{Eq:omegaistensorial}
&R_{a}^*\omega=\mathsf{Ad}_{a^{-1}}\omega,\quad\text{for all}\quad a\in GL_n(\mathbb{R}).
\end{align}
It is well known that if $\Omega$ and $\Theta$ are respectively the curvature and torsion forms of the connection $\omega$,  they   satisfy Cartan's structure equations 
\begin{align*}
d\omega(X,Y)&=-\dfrac{1}{2}\left[\omega(X),\omega(Y)\right]+\Omega(X,Y),\\ 
d\theta(X,Y)&=-\dfrac{1}{2}\left(\omega(X)\cdot\theta(Y)-\omega(Y)\cdot\theta(X)\right)+\Theta(X,Y),
\end{align*}
for all $X,Y\in T_u(P)$, $ u\in P$ (see \cite{KN} page 75). 

Given the presheaf $U\rightarrow \mathsf{aff}(\mathbb{R}^n)$, where $U$ is an open set in $P$, consider the  sheaf $\mathcal{F}$ (called simple faisceau by the French school, see \cite{G}, see also \cite{Chr}) of base $P$ and fibre the Lie algebra $\mathsf{aff}(\mathbb{R}^n)$ generated by the presheaf so that the restriction operations are reduced to the identity.

\begin{lemma}  The simple sheaf $\mathcal{F}$ acts on $\mathbb{R}^n$ by 
\begin{equation} \label{EQ:infinitesimalaction} (u,s(u))\cdot w=(u,(v_u,f_u))\cdot w:=v_u+f_u(w) \end{equation}
Moreover, if $\eta:Sec(TP)\longrightarrow Sec(P\times \mathsf{aff}(\mathbb{R}^n))$ is a homomorphism of sheaves of Lie algebras, then $Sec(TP)$ acts  on $\mathbb{R}^n$ via $\eta$.	\end{lemma}
\begin{proof} Recall the construction of the \'etale space corresponding to $\mathcal{F}$. As $\mathcal{F}(u)$ is a direct limit we have $\mathcal{F}(u)=\mathsf{aff}(\mathbb{R}^n)$ for any $u\in P$.

Moreover the sections over $U$ considered as maps from $U$ to $\mathsf{aff}(\mathbb{R}^n)$ are locally constant. Hence the topology of $\mathcal{F}=P\times \mathsf{aff}(\mathbb{R}^n)$ is the product topology of that of $P$ by the discrete topology of $\mathsf{aff}(\mathbb{R}^n)$.

As a consequence, Equation \eqref{EQ:infinitesimalaction} can be written locally as 
\[ s\cdot w=(v,f)\cdot w=v+f(w). \] 
This  gives an infinitesimal action of  $\mathcal{F}$ on $\mathbb{R}^n$.

The last assertion easily follows.
\end{proof}

As a consequence the Lie algebras $Sec(TP),$ $\mathfrak{a}(P)$ and $\mathfrak{a}_c(P)$ also act infinitesimally on $\mathbb{R}^n$.

In these terms we have the following result.

\begin{theorem} \label{T:infinitesimalversion} Let $M$ be a smooth connected manifold, $P=L(M)$ be the principal bundle of linear frames of the manifold, $\theta$ the fundamental form of $P$ and $\Gamma$ a linear connection on $P$ of connection form $\omega$. The following assertions are equivalent
\begin{enumerate}[(i)]
\item The linear connection $\Gamma$ on $P$  is flat affine.
\item The map
$$
\begin{array}{ccc}
\eta:Sec(TP)&\longrightarrow &Sec(P\times\mathsf{aff}(\mathbb{R}^n))\\	Z  &\mapsto         & (\theta(Z),\omega(Z))
\end{array},	$$
is a homomorphisms of sheaves of real Lie algebras.
\end{enumerate} 
\end{theorem}\begin{proof} As $\mathsf{aff}(\mathbb{R}^n)=\mathbb{R}^n\rtimes_{id}\mathsf{gl}(\mathbb{R}^n)$ is the semidirect product of the abelian Lie algebra $\mathbb{R}^n$ and the Lie algebra of commutators of linear endomorphisms of $\mathbb{R}^n$, the map 
$$
\begin{array}{rccl}
\eta:& Sec(TP)&\longrightarrow&Sec(P\times\mathsf{aff}(\mathbb{R}^n))\\
&           Z            &\mapsto    &(\theta(Z),\omega(Z))
\end{array}
$$
is well defined and $\mathbb{R}$-linear.

That $\eta$ is a homomorphism of sheaves of Lie algebras means that
\begin{equation} \label{Eq:homomorphismeta} \eta([Z_1,Z_2])=(\omega(Z_1)\cdot\theta(Z_2)-\omega(Z_2)\cdot\theta(Z_1),[\omega(Z_1),\omega(Z_2)]),\text{ for all }Z_1,Z_2\in \text{Sec}(TP). \end{equation}
In other words, $\omega$ is a linear representation of the Lie algebra $Sec(TP)$ and  $\theta$ is a 1-cocycle relative to this representation.

Suppose that $\Gamma$ is flat and torsion free, then the structure equations for $\omega$ reduce to
\begin{align}\label{Eq:reducedestructureequation1} d\theta(Z_1,Z_2)&=
-\frac{1}{2}(\omega(Z_1)\cdot\theta(Z_2)-\omega(Z_2)\cdot\theta(Z_1))\\ \label{Eq:reducedestructureequation2}
d\omega(Z_1,Z_2)&=
	-\frac{1}{2}[\omega(Z_1),\omega(Z_2)]. \end{align}
It is clear that  \eqref{Eq:homomorphismeta} holds when both vector fields are vertical or if one is vertical and  the other is horizontal.

Now let us suppose that both $Z_1$ and $Z_2$ are horizontal vector fields. As the distribution determined by $\omega$ is integrable, it follows that $[Z_1,Z_2]$ is horizontal and  Equation \eqref{Eq:reducedestructureequation2} implies that $\omega([Z_1,Z_2])=[\omega(Z_1),\omega(Z_2)]$. Also notice that $\omega(Z_1)\cdot\theta(Z_2)-\omega(Z_2)\cdot\theta(Z_2)=0.$ Now, we can suppose that $Z_1$ and $Z_2$ are basic vector fields, that is, $Z_1=B(\xi_1)$ and $Z_2=B(\xi_2)$. Hence  from \eqref{Eq:reducedestructureequation1} we get  that $\theta([Z_1,Z_2])=0$. Therefore assertion (ii) follows from (i).
 

To prove that (ii) implies (i), we will show that  Equation \eqref{Eq:homomorphismeta} implies that the curvature and the torsion vanish, i.e., Equations \eqref{Eq:reducedestructureequation1} and \eqref{Eq:reducedestructureequation2} hold.

As $\eta$ is a homomorphism of sheaves of Lie algebras it follows that for $Z_1, Z_2\in Sec(TP)$ 
	\[ \omega([Z_1,Z_2])=[\omega(Z_1),\omega(Z_2)]\quad\text{ and}\quad \theta([Z_1,Z_2])=\omega(Z_1)\cdot\theta(Z_2)-\omega(Z_2)\cdot\theta(Z_1)\] 
	
The first equality means that the local horizontal distribution on $P$ determined by $\omega$ is  completely integrable, hence  $\Omega$ vanishes.

On the other hand, as the torsion $\Theta$ is a tensorial 2-form on $P$ of type $(GL_n(\mathbb{R}),\mathbb{R}^n)$, we have $\Theta(Z_1,Z_2)=0$ if $Z_1$ or $Z_2$ is vertical. Now,  if  $Z_1=B(\xi_1)$ and $Z_2=B(\xi_2)$  for some $\xi_1,\xi_2\in\mathbb{R}^n$, we obtain
\begin{eqnarray*}
\Theta(Z_1,Z_2) & = & d\theta(Z_1,Z_2)\\ & = & \frac{1}{2}\left(Z_1(\theta(Z_2))-Z_2(\theta(Z_1))-\theta([Z_1,Z_2])\right)\\ & = & \frac{1}{2}\left(Z_1(\xi_2)-Z_2(\xi_1)-\omega(Z_1)\cdot\theta(Z_2)+\omega(Z_2)\cdot\theta(Z_1)\right)\\ & = & 0.
\end{eqnarray*}
Consequently the torsion form $\Theta$ vanishes. 	
\end{proof}

Recall that there are topological obstructions to the existence of a  flat affine connection. In particular, a closed manifold with finite fundamental group does not admit flat affine structures (see \cite{Aus}, \cite{JM} and \cite{S}).

As $P=L(M)$ has a natural parallelism, the group $K$ of transformations of this parallelism is a Lie group of dimension at most  $\dim P=n^2+n$. More precisely, given $\sigma\in K$ and  any $u\in P$, the map $\sigma\mapsto \sigma(u)$ is injective and its image $\{\sigma(u)\mid \sigma\in K\}$ is a closed submanifold of $P$. The submanifold structure of $\{\sigma(u)\mid \sigma\in K\}$ turns $K$  into a Lie transformation group of $P$ (see \cite{SK}).  The group $Aut(P,\omega)$, of diffeomorphisms of $P$ preserving $\omega$ (therefore $\theta$), is a  closed Lie subgroup of $K$, so it is a Lie group of transformations of $P$.


\begin{remark} \label{R:affineatlas}
The group $Aut(L(\mathbb{R}^n),\omega^0)$  of diffeomorphisms of $L(\mathbb{R}^n)$ preserving the usual connection $\omega^0$ on $\mathbb{R}^n$ is isomorphic to the  group $\mathsf{Aff}(\mathbb{R}^n,\nabla^0)$, called the classic affine group, that is, the group generated by translations and linear automorphisms. Recall also that  having a flat affine structure $\nabla$ on $M$ is equivalent to have a smooth atlas whose change of coordinates are elements of $Aut(L(\mathbb{R}^n),\omega^0)$.	
\end{remark}

The following technical result is a direct  consequence of Theorem \ref{T:affineetalerepresentationofmanifolds}, it  concerns certain $G$-structures to which a connection can be attached. 
\begin{corollary} Let $\Gamma$ be flat affine  and  $\eta$ the homomorphism of  Theorem \ref{T:infinitesimalversion}, then we have
	\begin{enumerate}[a.]
\item The connection $\Gamma$ is a metric connection, if and only if  $\eta$  takes values in the Lie algebra $e_{(p,q)}(n)=\mathbb{R}^n\rtimes o_{(p,q)}$.
\item  The connection $\Gamma$ preserves a volume form if and only if  $\eta$  takes values in the Lie algebra $\mathbb{R}^n\rtimes sl_n(\mathbb{R})$. 
\noindent 
In particular, if $n=2m$, $\Gamma$ preserves a symplectic form  $\sigma$ if and only if   $\eta$  takes values in the Lie algebra $\mathbb{R}^{n}\rtimes sp_{m}(\mathbb{R})$.\end{enumerate}
\end{corollary}
\begin{proof} Consider the $G$-structure on $M$, whenever exists,  given  respectively by $O_{(p,q)}$ (with $p+q=n$), $SL_n(\mathbb{R})$ and $Sp_{m}(\mathbb{R})$. In each respective case, the homomorphism $\eta$ takes values in the sheaf of Lie algebras $Sec(P\times (\mathbb{R}^n\rtimes o_{(p,q)}))$, $Sec(P\times (\mathbb{R}^n\rtimes sl_n(\mathbb{R}))$ and  $Sec(P\times (\mathbb{R}^n\rtimes sp_{m}(\mathbb{R}))$, respectively. 
\end{proof}

\begin{remark} 
If $\Gamma$ is flat affine and $X\in\mathfrak{a}(M,\nabla)$ with flow $\phi_t$, then its natural lift $L(\phi_t)$ preserves the foliation defined by $\omega$. Moreover, if $L(X)$ is complete then $L(\phi_t)$ is an automorpism of the bi-foliated manifold $L(M)$. 
\end{remark}

The groups $Aut(P,\omega)$ and  $\mathsf{Aff}(M,\nabla)$ are  isomorphic. We denote by  $Aut(P,\omega)_0$ the unit component of $Aut(P,\omega)$ and by $\widehat{Aut}(P,\omega)_0$ its universal covering Lie group.

\begin{example} Identify $\mathbb{C}$ with the plane $\mathbb{R}^2$ and $\mathbb{C}^*$ with the punctured plane $\mathbb{R}^2\setminus\{(0,0)\}$ and let $p:\mathbb{R}^2\longrightarrow\mathbb{R}^2\setminus\{(0,0)\}$ be the complex exponential map. It is clear that $p$ is a covering map with automorphism group given by $H=\{F_k:\mathbb{R}^2\longrightarrow \mathbb{R}^2\mid F_k(x,y)=(x,y+2\pi k)\ \text{with}\ k\in\mathbb{Z} \}$. As $H$ is a subgroup of affine transformations of $\mathbb{R}^2$ relative to $\nabla^0$,  there exists a unique linear connection $\nabla$ on $M:=\mathbb{R}^2\setminus\{(0,0)\}$ so that $p$ is an affine map. Consequently  the $\sf{Aff}(\mathbb{R}^2)-$principal fiber bundle of affine frames $A(M)$ has a connection $\Gamma'$, of curvature zero, determined by $\nabla$. The geodesic curves in $M$ determined by $\nabla$ are logarithmic spirals. The elements of the group $\mathsf{Aff}(M,\nabla)$, of affine transformations of $M$ relative to $\nabla$, are diffeomorphisms of $M$ locally given by $F(r,\theta)=(ar^b,c\ln(r)+\theta+d )$ where $a,b,c,d\in \mathbb{R}$ with $a>0$ and  $(r,\theta)$ are polar coordinates of $M$. 	
	
Let $u_0'$ be in $A(M)$ and $H(u_0')$ be the affine holonomy bundle through $u_0'$. Let us denote by $h$  affine holonomy representation of $(M,\nabla)$. It is well known that $H(u_0')$ is a covering manifold of $M$ with structure group $h(\pi_1(M))$ and that $A(M)$  is the disjoint union of holonomy bundles. Let $p':H(u_0')\longrightarrow M$ be the covering map and $\widetilde{F}$ the natural affine lift   of $F\in \mathsf{Aff}(M,\nabla)$ to $A(M)$, i.e., $\widetilde{F}((q,X_1,X_2)_{(x,y)})=((F_{*,(x,y)}(q),F_{*,(x,y)}(X_1),F_{*,(x,y)}(X_2))_{F(x,y)})$, with $F_{*,(x,y)}(q)$  the end point of the vector $F_{*,(x,y)}(\overrightarrow{q})$ and $\overrightarrow{q}$ the vector on $T_{(x,y)}$ from the origin to $q$. It can be verified that $\widetilde{F}$ sends the affine holonomy bundle $H(u_0')$ through $u_0'$  to the holonomy bundle $H(\widetilde{F}(u_0'))$ through $\widetilde{F}(u_0')$. 
\end{example}

Recall that $\omega$ is flat affine if and only if its affine holonomy group is discrete (see \cite{KN} p 210). Nevertheless, there are flat affine manifolds whose linear holonomy group is not discrete.


The Lie algebra homomorphism $\eta$ of Theorem \ref{T:infinitesimalversion} could be integrated, in some particular cases, into a Lie group homomorphism. For instance,  applying Lie's third Theorem (see \cite{Che} for details on the proof of this theorem) and Cartan's Theorem to the Lie algebra homomorphism $\eta':\mathfrak{a}_c(P)\longrightarrow \mathsf{aff}(\mathbb{R}^n)$, obtained from $\eta$ restricted to $\mathfrak{a}_c(P)$, we have.

\begin{corollary} \label{C:integrationofeta} Let $(M,\nabla)$ be a flat affine connected manifold    of connection form $\omega$. Then there exists a Lie group homomorphism $\rho:\widehat{Aut}(P,\omega)_0\longrightarrow Aut(L(\mathbb{R}^n),\omega^0)$ determined by $(\theta,\omega)$.
\end{corollary}

\begin{lemma}\label{L:surjectivetheta} Let $M$ be a smooth connected manifold and $\Gamma$ a flat affine connection on $P=L(M)$ with $\dim(\mathfrak{a}_c(P))\geq\dim(M)$. Then there exists $u\in P$ so that the map $\phi_u:\mathfrak{a}_c(P)\longrightarrow \mathbb{R}^n$ defined by $\phi_u(Z)=\eta(0)(Z)=\theta_u(Z_u)$ is onto.
\end{lemma}
\begin{proof} Let $\beta=(Z_1,\dots,Z_m)$ be a linear basis of $\mathfrak{a}_c(P)$.  We claim that there are linear frames $u$ where the set $(\eta_u(0)(Z_1)=\theta_u(Z_{1,u}),\dots,(\eta_u(0)(Z_m)=\theta_u(Z_{m,u}))$ spans $\mathbb{R}^n$. Otherwise, for every $u$ and every subset $(X_1=Z_{i_1},\dots,X_n=Z_{i_n})$ of $\beta$ with $n=\dim(M)$ and $1\leq i_1<\dots<i_n\leq m$, the set  $\{\theta_u(X_{1,u}),\dots,\theta_u(X_{n,u})\}$ is linearly dependent. Thus, there exists real constants $\alpha_1,\dots,\alpha_n$  so that $0=\sum\alpha_i \theta_u(X_{i,u})$. It follows that
\[ 0=\sum\alpha_i \theta_u(X_{i,u})=\theta_u\left(\sum\alpha_iX_{i,u}\right)=u^{-1}(\pi_{*,u}(X_u) ) \]
where $u$ is seen as a linear transformation from $\mathbb{R}^n$ to $T_{\pi(u)}M$ and  $X=\sum\alpha_iX_{i}$. It follows that $\pi_{*,u}(X_u)=0$ for every $u$, that is, $X\equiv 0$. But this contradicts the fact that the set $\{X_1,\dots,X_n \}$ is linearly independent. As the map $\phi_u$ is linear, we conclude that it is onto.	
\end{proof}

The following example exhibits a flat affine manifold $M$ whose bundle $P=L(M)$ has frames $u$ where the map $\phi_u$ of the previous lemma is not onto.

\begin{example} Consider the manifold $M=\mathbb{R}^2\setminus\{(0,0),(0,1)\}$ with the connection $\nabla$ induced by $\nabla^0$. A simple calculation shows that the real space $\mathsf{aff}(M,\nabla)$ is two dimensional with linear basis given by $\left(X_1=x\frac{\partial}{\partial x},X_2=x\frac{\partial}{\partial y} \right)$. One can also verify that, for a linear frame $u$ over a point $(x,y)\in M$, one has that $\theta_u(\widetilde{X}_1)=u^{-1}\left[ \begin{matrix} x\\ 0
\end{matrix} \right]$ and $\theta_u(\widetilde{X}_2)=u^{-1}\left[ \begin{matrix} 0\\ x
\end{matrix} \right]$, where $\widetilde{X}_1$ and $\widetilde{X}_2$ are the natural lifts of $X_1$ and $X_2$, respectively. Notice that the set $\{ \theta_u(\widetilde{X}_1), \theta_u(\widetilde{X}_2)\}$ is linearly independent if and only if $x\ne0$. 
\end{example}

\begin{remark} \label{R:openorbitforLiegroups} If $M=G$ is a Lie group endowed with a flat affine left invariant connection $\nabla^+$, then the map $\phi_u$ of the previous lemma is onto for every $u\in P$. This follows from the fact that right invariant vector fields are complete infinitesimal affine transformations relative to $\nabla^+$ forming a global parallelism of $G$.
\end{remark}

\begin{theorem}\label{T:affineetalerepresentationofmanifolds} Let $M$ be a smooth connected $n$-dimensional manifold and $\Gamma$ a flat affine connection on $P=L(M)$ of connection form $\omega$. If $dim(Aut(P,\omega))\ge n$, there exists a homomorphism $\rho:\widehat{Aut}(P,\omega)_0\longrightarrow Aut(\mathbb{R},\omega^0)$ having a point of open orbit. Moreover, if $dim(Aut(P,\omega))=n$, the point also has discrete isotropy, that is, $\rho$ is an \'etale affine representation of  $\widehat{Aut}(P,\omega)_0$ in $\mathbb{R}^n$.  
\end{theorem}
\begin{proof}
Let $u\in P$ be a linear frame,  $\phi_u$ as in the previous lemma, $\eta':\mathfrak{a}_c(P)\longrightarrow \mathsf{aff}(\mathbb{R}^n)$ the map defined by $\eta'(Z)=(\theta_u(Z),\omega_u(Z))$ and $\rho:\widehat{Aut}(P,\omega)\longrightarrow Aut(L(\mathbb{R}^n),\omega^0)$  the homomorphism satisfying that $\rho_{*,Id}=\eta'$. If $\pi:\widehat{Aut}(P,\omega)\longrightarrow Orb_0$ is the orbital map, since $\pi_{*,Id}(Z)=\eta(0)(Z)$, it follows that $\pi_{*,Id}=\phi_u$. As $\phi_u$ is onto, we get that  0 is a point of open orbit. The second assertion in immediate. 
\end{proof} 

The following result is a direct consequence of the preceding theorem.

\begin{corollary} Under the hypothesis of Theorem \ref{T:affineetalerepresentationofmanifolds}, manifolds $M$ satisfying the condition $dim(\mathsf{Aff}(M,\nabla))=\dim(M)$ give a positive answer to Question \ref{Q:Medinasquestion}. That is, the group $\mathsf{Aff}(M,\nabla)$ is endowed with a flat affine left invariant connection determined by $\nabla$.
\end{corollary}

Recall that there are flat affine manifolds $(M,\nabla)$ with dimension equal to $\dim(\mathsf{Aff}(M,\nabla))$. For instance, flat affine tori other than the Hopf torus listed in \cite{Nag} (see also \cite{Ben}) and one of the manifolds of Example \ref{Ex:calculationofaffinegroups} below.

Notice that the proof of Lemma \ref{L:surjectivetheta}, implies that if $\dim(\mathfrak{a}_c(P))<n$, the map $\phi_u$ is injective, so we have

\begin{corollary} Under the conditions of Theorem \ref{T:affineetalerepresentationofmanifolds}, if $\dim(\mathfrak{a}_c(P))<n$, there exists an injective homomorphism $\rho:\widehat{Aut}(P,\omega)_0\longrightarrow Aut(L(\mathbb{R}^n,\omega^0 ))$, that is, a faithful representation of $\widehat{Aut}(P,\omega)_0$ by classical affine transformations of $\mathbb{R}^n$. 
\end{corollary} 
\begin{proof} Let $\eta'$ be  as in the proof of Theorem \ref{T:affineetalerepresentationofmanifolds} and suppose that $\eta'(Z)=\eta'(Y)$, for some $Y,Z\in \mathfrak{a}_c(P)$, hence $\theta_u(Z_u)=\theta_u(Y_u)$. As $\theta_u$ is inyective for some $u\in P$, we have that $Y_u=Z_u$. Since $Y$ and $Z$ are natural lifts of infinitesimal affine transformations, it follows that $Y=Z$. 
\end{proof}

\begin{remark} There are flat affine manifolds $(M,\nabla)$ of dimension greater than $\dim(\mathsf{Aff}(M,\nabla))$. This condition is satisfied by some of the flat affine Klein bottles listed in \cite{FA}. More precisely, flat affine Klein bottles that are not double covered by a Hopf torus.
\end{remark}
\begin{example} \label{Ex:calculationofaffinegroups} Let  $M_1$, $M_2$ and $M_3$  be the plane without one, two and three non colineal points, respectively, and let  $\nabla_i$, $i=1,2,3$, be the connection $\nabla^0$ restricted to $M_i$.  Let us suppose that the  points removed are $p_1=(0,0)$, $p_2=(0,1)$ and $p_3=(1,0)$. An easy calculation shows that the corresponding groups of affine transformations of $M_i$ relative to $\nabla_i$ are given by
\[ \mathsf{Aff}(M_1,\nabla_1)=\{  F:M_1\longrightarrow M_1\mid F(x,y)=(ax+by,cx+dy)\quad\text{such that}\quad ad-bc\ne0\}.     \]
\[ \mathsf{Aff}(M_2,\nabla_2)=\{  F\in \mathsf{Diff}(M_2)\mid F(x,y)=(ax,bx+y)\text{ or }  F(x,y)=(ax,bx-y+1), \ a\ne0\}    \]
notice that this is a 2-dimensional non-commutative group with four connected components.

The group $\mathsf{Aff}(M_3,\nabla_3)$ is discrete, its elements are the set of affine transformations of $(\mathbb{R}^2,\nabla^0)$ permuting the points  $p_1,\ p_2$ and $p_3$.

The group of diffeomorphisms $Aut(P,\omega_1)$ of $P=L(M_1)$ preserving the connection form $\omega_1$ is given by
\[ \left\{F\in \sf{Diff}(P)\ \bigg|\ \begin{array}{lc} F\left(x,y,X_{11},X_{12},X_{21},X_{22}\right) =\\\left(ax+by,cx+dy,aX_{11}+bX_{21},aX_{12}+bX_{22},cX_{11}+dX_{21},cX_{12}+dX_{22}\right)\end{array} \right\}  \]
	where $(x,y,X_{11},X_{12},X_{21},X_{22})$ are local coordinates of $L(M_1)$.
	
The connected component of the unit of the group $ Aut(L(M_2),\omega_2)$, of automorphisms of $P=L(M_2)$ preserving the connection form $\omega_2$ associated to $\nabla_2$, is formed by diffeomorphisms $F_{a,b}$, with $a>0$,  given by
\[  F_{a,b}\left(x,y,X_{11},X_{12},X_{21},X_{22}\right) =\left(ax,bx+y,aX_{11},aX_{12},bX_{11}+X_{21},bX_{12}+X_{22}\right),  \]
where $(x,y,X_{11},X_{12},X_{21},X_{22})$ are local coordinates of $L(M_2)$.


The homomorphism of Corollary \ref{C:integrationofeta} corresponding to the manifold $M_1$ is given by 
\[ \begin{matrix}\rho_1:&\mathsf{Aff}(M_1,\nabla_1)&\rightarrow&\mathsf{Aff}(\mathbb{R}^2)\\ 
& F=F_{(Y_1,Y_2,Y_3,Y_4)}&\mapsto& \rho_1(F)\end{matrix} \]
where the linear part of  $\rho_1(F)$ is given  by
\[ 
\left[ \begin{matrix} \frac{X_{11}X_{22}Y_1-X_{12}X_{21}Y_4-X_{11}X_{12}Y_3+X_{21}X_{22}Y_2}{X_{11}X_{22}-X_{12}X_{21}} & \frac{X_{22}^2Y_2-x_{12}^2Y_3+(Y_1-Y_4)X_{12}X_{22}}{X_{11}X_{22}-X_{12}X_{21}}\\   \ &\ \\ \frac{X_{11}^2Y_3-X_{21}^2Y_2+(Y_4-Y_1)X_{11}X_{21}}{X_{11}X_{22}-X_{12}X_{21}}& \frac{X_{11}X_{22}Y_4-X_{12}X_{21}Y_1+X_{11}X_{12}Y_3-X_{21}X_{22}Y_2}{X_{11}X_{22}-X_{12}X_{21}} 
	\end{matrix}  \right], \]
and the cocycle part of $\rho_1(F)$ is 
\[ \left[\begin{matrix}  \frac{u_2X_{12}-u_1X_{22}+u_1X_{22}Y_1-u_2X_{12}Y_4-u_1X_{12}Y_3+u_2X_{22}Y_2}{X_{11}X_{22}-X_{12}X_{21}} \\ \ \\ \frac{u_1X_{21}-u_2X_{11}-u_1X_{21}Y_1+u_2X_{11}Y_4+u_1X_{11}Y_3-u_2Xx_{21}Y_2}{X_{11}X_{22}-X_{12}X_{21}} \end{matrix}\right] \]
with  $(u_1,u_2)$ are local coordinates on $M_1$, $(X_{11},X_{12},X_{21},X_{22})$ are local coordinates for $GL_2(\mathbb{R})$ and  $F= F_{(Y_1,Y_2,Y_3,Y_4)}$ defined by $  F(u_1,u_2)=(Y_1u_1+Y_3u_2,Y_2u_1+Y_4u_2)$ with $Y_1Y_4-Y_2Y_3\ne0$.  Moreover, Theorem \ref{T:affineetalerepresentationofmanifolds} guarantees that the action of $\mathsf{Aff}(M_1,\nabla_1)$ on $\mathbb{R}^2$ determined by $\rho_1$ leaves and open orbit. In fact, the open orbit is the whole  plane.
	
Now,  a direct calculation shows that the homomorphism of  Corollary \ref{C:integrationofeta} corresponding to the manifold $M_2$ is given by $$\begin{matrix}\rho:&G(\mathfrak{a}_c(P))&\longrightarrow& \mathsf{Aff}(\mathbb{R}^2)\\ &F_{a,b}&\mapsto&\rho(F_{a,b})\end{matrix}$$  where $G(\mathfrak{a}_c(P))$ is the connected and simply connected Lie group of Lie algebra $\mathfrak{a}_c(P)$, $F_{a,b}$ the affine map of $(M_2,\nabla_2)$ defined by $F_{a,b}(x,y)=(ax,bx+y)$ with $a>0$ and  $\rho(F_{a,b})$ is given by
\[  \frac{1}{D}\left[\begin{matrix} aX_{11}X_{22}-bX_{11}X_{12}-X_{12}X_{21}&(a-1)X_{12}X_{22}-bX_{12}^2&(a-1)xX_{22}-bxX_{12}\\(1-a)X_{11}X_{21}+bX_{11}^2& -aX_{12}X_{21}+bX_{11}X_{12}+X_{11}X_{22}&(1-a)xX_{21}+bxX_{11}\\0&0&D         \end{matrix}\right]  \] 
where   $(X_{11},X_{12},X_{21},X_{22})$ is a system of local coordinates of $GL_2(\mathbb{R})$,   $D=X_{11}X_{22}-X_{12}X_{21}$ and $(x,y)$ local coordinates of $M_2$. According to Theorem \ref{T:affineetalerepresentationofmanifolds}, the representation $\rho$ is \'etale whenever $x\ne0$, hence, $\mathsf{Aff}(M_2,\nabla_2)$ admits a flat affine (left invariant) connection. 
\end{example}

If $M=G$ is a Lie group, we will identify its Lie algebra  Lie$(G)=\mathfrak{g}$ with the real vector space of left invariant vector fields on $G$ and also with the tangent space at the unit $\varepsilon\in G$. For any $x\in T_\varepsilon G$, we denote by $x^+$ (respectively by $x^-$) the left (right) invariant vector field determined by $x$.
The group  $G$ acts on itself on the left (respectively right), and we will denote by $L_\sigma$ (respectively $R_\sigma$) the left (respectively right) action of  $\sigma\in G$. These actions naturally lift to actions of $G$ on $P=L(G)$ given by $\psi_{1,\sigma}(u)=\sigma\cdot u= \left(\sigma\tau,(L_\sigma)_{*,\tau}(X_{ij})\right)$ (respectively   $\psi_{2,\sigma}(u)=u\cdot \sigma=(\tau\sigma,(R_\sigma)_{*,\tau}(X_{ij}))$
where $u=(\tau,(X_{ij}))$. A linear connection $\Gamma$ on $G$ is said left invariant  (respectively right invariant) if the action $\psi_1$ (respectively  $\psi_2$) preserves the  horizontal distribution. The connection is bi-invariant if the  horizontal distribution is preserved by  $\psi_1$ and $\psi_2$. 



The existence of a left invariant flat affine connection on a Lie group $G$, is equivalent to the existence of an \'etale affine representation of $G$, i.e, there exists a homomorphism of Lie groups $\tau:G\longrightarrow \mathsf{Aff}(\mathbb{R}^n)$ having at least an open orbit with discrete isotropy. 


In these terms we get the following consequence of Theorem \ref{T:infinitesimalversion}  and Corollary \ref{C:integrationofeta}. In the next result the Lie group $G$  can be supposed  simply connected
 
\begin{theorem} \label{T:rhointheLiegroupscase} Let $G$ be a connected $n$-dimensional Lie group, of Lie algebra $\mathfrak{g}$, endowed with a left invariant connection $\nabla^+$ of connection form $\omega^+$. The following assertions are equivalent
\begin{enumerate}	\item[1] The connection $\omega^+$ is flat affine.
\item[2] There exists a unique Lie group homomorphism  $\rho:\widehat{Aut}(L(G),\omega^+)_0\longrightarrow Aut(L(\mathbb{R}^n),\omega^0)$ with derivative  $\rho_{*,Id}=(\theta,\omega^+)$ leaving an open orbit. 
\item[3] There exists an affine \'etale representation $\rho':G^{op}\longrightarrow\mathsf{Aff}(\mathbb{R}^n,\nabla^0)$ determined by  $(\theta,-\omega^+)$. 
\end{enumerate}  \end{theorem}\begin{proof} That 1. implies 2. follows from  Corollary \ref{C:integrationofeta}, Theorem \ref{T:affineetalerepresentationofmanifolds} and Remark \ref{R:openorbitforLiegroups}.



Now, suppose that there is a Lie group homomorphism $\rho:\widehat{Aut}(L(G),\omega^+)_0\longrightarrow Aut(L(\mathbb{R}^n),\omega^0)$ with a point of open orbit so that $\rho_{*,Id}=(\theta,\omega^+)$. That is, $\rho_{*,Id}=\eta_{/\mathfrak{a}_c(P)}$ where $\eta$ is the homomorphism of Theorem \ref{T:infinitesimalversion}. As right invariant vector fields on $G$ are infinitesimal affine transformations, we have that $\mathfrak{g}^{op}$, the opposite Lie algebra of $\mathfrak{g}$, can be considered as a Lie subalgebra of $\mathfrak{a}_c(L(G))$. Therefore the restriction $\eta'=\eta/_{\mathfrak{g}^{op}}:\mathfrak{g}^{op}\longrightarrow \mathsf{aff}(\mathbb{R}^n)$ is also a Lie algebra homomorphism. Using Cartan's and Lie's third Theorems we get a Lie group homomorphism $\rho':G^{op}\longrightarrow  \mathsf{Aff}(\mathbb{R}^n)$. Let $\pi':G^{op}\longrightarrow Orb(0)$ be the orbital map. It is easy to verify that $\pi'_{*,\epsilon}(Z)=\theta_u(Z)$, hence by mimicking  the proof of Lemma \ref{L:surjectivetheta}, we get that $\pi'_{*,\epsilon}(Z)$ is a linear isomorphism, therefore $\rho'$ is \'etale. Hence 2. implies 3. is true.

Finally suppose 3., i.e., there exists an affine \'etale representation $\rho':G^{op}\longrightarrow\mathsf{Aff}(\mathbb{R}^n,\nabla^0)$ with $\rho'_{*,\epsilon}=(\theta,-\omega^+)$. Then it is easy to verify that the map $\lambda:\mathfrak{g}\longrightarrow\mathsf{aff}(\mathbb{R}^n,\nabla^0)$ defined by $\lambda=(\theta,\omega^+)$ is a representation of $\mathfrak{g}$. By Lie's third theorem, the map $\rho:G\longrightarrow\mathsf{Aff}(\mathbb{R}^n,\nabla^0)$ with $\rho_{*,\epsilon}=\lambda$ is an affine \'etale representation of $G$.  
	Therefore $\omega^+$ is a flat affine left invariant connection on $G$.  
\end{proof}


\section{The associative envelope of the Lie algebra of a Lie group of flat affine transformations} \label{S:associativeenvelope}

Given a flat affine manifold $(M,\nabla)$ and a Lie subgroup $H$ of the group $\mathsf{Aff}(M,\nabla)$, we show the existence of a simply connected  Lie group endowed with a flat affine bi-invariant connection whose Lie algebra contains  $Lie(H)$ as a Lie subalgebra (see Theorem \ref{T:envelopeofaflataffinemanifold}). 

For this purpose we recall some known facts and introduce some notation. Let $M$ be an $n$ dimensional manifold with a linear connection $\Gamma$ and corresponding covariant derivative   $\nabla$, consider the product \begin{equation}\label{Eq:productinducedbyaconnection}XY:=\nabla_XY,\quad\text{for}\quad X,Y\in \mathfrak{X}(M).\end{equation} 
If the curvature  tensor relative to $\nabla$ is identically zero, this product verifies the condition
\[ [X,Y]Z=X(YZ)-Y(XZ), \quad\text{for all}\quad X,Y,Z\in \mathfrak{X}(M).\]
If both torsion and curvature vanish identically we have
\begin{equation} \label{Eq:leftsymmetricproduct} (XY)Z-X(YZ)=(YX)Z-Y(XZ), \quad\text{for all}\quad X,Y,Z\in \mathfrak{X}(M).\end{equation} 
A vector space endowed with a bilinear product satisfying Equation \eqref{Eq:leftsymmetricproduct} is called a left symmetric algebra (see \cite{Vin}, see also \cite{JLK}). These algebras are also called Kozsul-Vinberg algebras or simply KV-algebras. The product given by $[X,Y]_1:=XY-YX$ defines a Lie bracket on the space. Moreover, if $\nabla$ is torsion free, this Lie bracket agrees with the usual Lie bracket of $\mathfrak{X}(M)$, that is \begin{equation}\label{equliebrakettorsionfree}
	[X,Y]=\nabla_XY-\nabla_YX
\end{equation}
Recall that product \eqref{Eq:productinducedbyaconnection} satisfies $(fX)Y=f(XY)$ and $X(gY)=X(g)Y+g(XY)$, for all $f,g\in C^\infty(M,\mathbb{R})$, i.e., the product defined above is $\mathbb{R}$-bilinear and $C^\infty(M,\mathbb{R})$-linear in the first component. 

 In what follows, given an associative or a left symmetric algebra $(\mathcal{A},\cdot)$, we will denote by $\mathcal{A}_-$ the Lie algebra of commutators of $\mathcal{A}$, i.e., the Lie algebra with bracket given by \[ [a,b]=a\cdot b-b\cdot a.\]
To our knowledge, the next result seems to appear for the first time in J. Vey's thesis (see \cite{V}, see also  \cite{Yag}).

\begin{lemma}\label{L:associativity} Let $(M,\nabla)$ be an $n$-dimensional flat affine manifold. Then Product \eqref{Eq:productinducedbyaconnection} turns the vector space $\mathfrak{a}(M,\nabla)$ into an  associative algebra, of dimension at most $n^2+n$,   whose commutator is the Lie bracket of vector fields on $M$.  In particular, if $(G,\nabla^+)$ is a flat affine Lie group of Lie algebra $\mathfrak{g}$, the space $\mathfrak{a}(G,\nabla^+)_-$   contains $\mathfrak{g}^{op}$  as a Lie subalgebra, where $\mathfrak{g}^{op}$ is the opposite Lie algebra of $\mathfrak{g}$.
\end{lemma}
\begin{proof}
Since $\nabla$ is flat affine, it follows from \eqref{Eq:ecuacionkobayashi} that a smooth vector field $X$ is an infinitesimal affine transformation if and only if 
\begin{equation} \label{Eq:infinitesimalaffinetransformationsonflataffineconnections} \nabla_{\nabla_YZ}X=\nabla_Y\nabla_ZX,\end{equation}
for all $Y,Z\in\mathfrak{X}(M)$. This equality  implies that $\nabla_XY\in\mathfrak{a}(M,\nabla),$ whenever $X,Y\in \mathfrak{a}(M,\nabla)$ and also gives that the product $X Y=\nabla_XY$ is associative. On the other hand, it follows from Equation   \eqref{equliebrakettorsionfree}  that the commutator of this product agrees with the Lie bracket of $\mathfrak{a}(M,\nabla)$. 
	
In particular if $M=G$ is a Lie group of Lie algebra $\mathfrak{g}$ and $\nabla^+$ is left invariant and flat affine,  the real vector space $\mathfrak{g}$ of right invariant vector fields on $G$ is a subspace of $\mathsf{aff}(G,\nabla^+)$. Hence from \eqref{equliebrakettorsionfree} we get  that $\mathfrak{g}^{op}$ is a Lie subalgebra of  $\mathfrak{a}(G,\nabla^+)_-$. 
\end{proof}


\noindent The following well known result will be used in what follows (see for instance \cite{M} or  \cite{BM}). 

\begin{proposition}\label{T:associativeimplybiinvariantstructures}  Let $G$ be a Lie group of Lie algebra $\mathfrak{g}$. The group $G$ admits a flat affine bi-invariant structure if and only if  $\mathfrak{g}$ is the underlying Lie algebra of an  associative algebra so that $[a,b]=ab-ba$, for all $a,b\in\mathfrak{g}$. 
\end{proposition}

We have the following consequence of Lemma \ref{L:associativity}.

\begin{theorem}\label{T:envelopeofaflataffinemanifold} Given a flat affine manifold $(M,\nabla)$ and a Lie subgroup $H$ of $\mathsf{Aff}(M,\nabla)$, there exists a  connected finite dimensional  Lie group $G$ endowed with a flat affine  bi-invariant structure containing a connected Lie subgroup  locally isomorphic to $H$.
\end{theorem}
\begin{proof} From Lemma \ref{L:associativity}, the real vector space $\mathfrak{a}(M,\nabla)$ is an associative algebra under the product determined by $\nabla$. Let $E$ be the  subalgebra of the associative algebra $\mathfrak{a}(M,\nabla)$ generated by the vector subspace $\mathfrak{h}=$Lie$(H)$ and $\widehat{\mathcal{A}}=E\oplus\mathbb{R}1$  the associative algebra obtained from $E$ by adjoining a unit element $1$. Consider $U(\widehat{\mathcal{A}})$ the group of units of $\widehat{\mathcal{A}}$, this is open and dense in $\widehat{\mathcal{A}}$. Let 
$$\overline{G}:=\left\{u\in U(\widehat{\mathcal{A}})\mid u=1+a,\ \text{with}\ a\in E  \right\}$$
and $G=\overline{G}_0$ the connected component of the unit of $\overline{G}$. Then the Lie group $G$ verifies the conditions of the theorem (for more details see \cite{BM}).

It is clear that $\mathfrak{h}$ is a Lie subalgebra of $E_-$. Consequently there exists a connected Lie subgroup $H'$ of $G$ of Lie algebra $\mathfrak{h}$, hence the groups $H$ and $H'$ are locally isomorphic. 
\end{proof}

\begin{remark} In general $\mathsf{aff}(M,\nabla)$ is not a subalgebra of the associative algebra $\mathfrak{a}(M,\nabla)$. This fact is contrary to what some authors state. This can be observed in Examples \ref{Ex:productoR1} and \ref{Ex:glnr}. 
\end{remark}

The previous remark and Theorem \ref{T:envelopeofaflataffinemanifold}  motivates the following.

\begin{definition}\label{D:associativeenvelope} If  $(M,\nabla)$ is a flat affine manifold and $H$ is a Lie subgroup  of $\mathsf{Aff}(M,\nabla)$, we set
\begin{enumerate}\item[(a)] The smallest subalgebra of the associative algebra $\mathfrak{a}(M,\nabla)$  containing the vector space $\mathfrak{h}=$Lie$(H)$,  denoted by env${}_\nabla(\mathfrak{h})$, will be called the associative envelope of  $\mathfrak{h}$ relative to $\nabla$.
\item[(b)] Any Lie group of Lie algebra env${}_\nabla(\mathfrak{h})$ will be said an enveloping Lie group of $H$.
\end{enumerate}	
\end{definition}

That any enveloping Lie group  is endowed with a flat affine bi-invariant connection determined by $\nabla$ follows from Proposition \ref{T:associativeimplybiinvariantstructures}.

\begin{corollary} Given a flat affine Lie group $(G,\nabla^+)$ of Lie algebra $\mathfrak{g}$, there exists a finite dimensional associative algebra $\mathcal{A}$ containing $\mathfrak{g}$ so that $\mathfrak{g}$ is a Lie subalgebra of $\mathcal{A}_-$.
\end{corollary}
\begin{proof} As $\mathfrak{g}^{op}$ is a subalgebra of the associative algebra $\mathfrak{a}(G,\nabla^+)$, there exists a Lie subgroup $H$ of $\mathsf{Aff}(G,\nabla^+)$ of Lie algebra $\mathfrak{g}^{op}$. By taking $\mathcal{A}=$env${}_\nabla(\mathfrak{h})^{op}$, the opposite algebra of env${}_{\nabla^+}(\mathfrak{h})$, we get an an associative algebra satisfying the statement.  
\end{proof}

\begin{definition}\label{D:envelopingofleftsymmetricalgebra} The algebra of the previous corollary, denoted by env${}_{\nabla^+}\mathfrak{(g)}$, will be called the associative envelope algebra of the left symmetric algebra $\mathfrak{g}$. 
\end{definition}

\begin{remark} Although the elements of the associative envelope $env_{\nabla^+}(\mathfrak{g})$ of the left symmetric algebra  $\mathfrak{g}=$Lie$(G)$ are differential operators of order less than or equal to 1 on $G$, the associative envelope is not a subalgebra of the universal enveloping algebra of $\mathfrak{g}$.
\end{remark}

\begin{example} \label{Ex:envelopingalgebraofthepuntcturedplane} Let $M=\mathbb{R}^2\setminus\{(0,0) \}$ and $\nabla$ be the connection on $M$ induced by the connection $\nabla^0$. That $G=\mathsf{Aff}(M,\nabla)_0$ is an enveloping  Lie group of itself results from the following observations. It can be verified that a linear basis for the real  space $\mathsf{aff}(M,\nabla)$ is given by $\left(x\dfrac{\partial}{\partial x},y\dfrac{\partial}{\partial x},x\dfrac{\partial}{\partial y},y\dfrac{\partial}{\partial y} \right)$. Also, this space is a subalgebra of the associative algebra $\mathfrak{a}(M,\nabla)$. 

As a consequence $G=\mathsf{Aff}(M,\nabla)$ is endowed with a flat affine bi-invariant connection determined by $\nabla$. 
\end{example}

\begin{example} \label{Ex:productoR1} 
Let $\nabla^+$ be the flat affine left invariant connection  on $G=\mathsf{Aff}(\mathbb{R})$ defined by
\begin{equation} \label{Eq:productodeterminadopornablacasoR1} \nabla^+_{e_1^+}e_1^+=2e_1^+,\qquad\nabla^+_{e_1^+}e_2^+=e_2^+,\qquad\nabla^+_{e_2^+}e_1^+=0,\quad \text{and}\quad\nabla^+_{e_2^+}e_2^+=e_1^+.\end{equation}
A direct calculation shows that a linear basis of $\mathfrak{a}(G,\nabla^+)$ is given by the following vector fields
\[ e_1^-=x\frac{\partial}{\partial x}+y\frac{\partial}{\partial y},\quad e_2^-=\frac{\partial}{\partial y},\quad C_3=\frac{1}{x}\frac{\partial}{\partial x},\quad C_4=\frac{y}{x}\frac{\partial}{\partial x},\quad  C_5=\left(x+\frac{y^2}{x}\right)\frac{\partial}{\partial x}\]\[\text{and}\qquad
C_6=\left(-xy-\frac{y^3}{x}\right)\frac{\partial}{\partial x}+(x^2+y^2)\frac{\partial}{\partial y},\]
where $e_1^-$ and $e_2^-$ denote the right invariant vector fields. As the connection is left invariant, the vector fields  $e_1^-$ and $e_2^-$ are complete. Moreover, it can be checked that no real linear combination of the fields $C_3,$ $C_4$, $C_5$ and $C_6$ is complete.  Consequently $(e_1^-,e_2^-)$ is a linear basis of $\mathsf{aff}(G,\nabla^+)=\mathfrak{g}^{op}$. 
	
On the other hand, the multiplication table of  the product defined by $\nabla^+$, i.e., the product $XY=\nabla_X^+Y$, on the basis of $\mathfrak{a}(G,\nabla^+)$ displayed above  is given by

\begin{equation} \label{Tab:tabla1}
\begin{array}{c|c|c|c|c|c|c}  & e_1^- &e_2^-&C_3&C_4&C_5&C_6\\\hline e_1^-&e_1^-+C_5& C_4&0&C_4&2C_5&2C_6\\\hline e_2^-& e_2^-+C_4&C_3&0&C_3&2C_4&2e_1^- -2C_5\\\hline C_3& 2C_3 & 0 &0&0&2C_3&2e_2^--2C_4\\\hline C_4&2C_4 &0 &0 &0&2C_4&2e_1^--2C_5\\\hline C_5&2C_5 &0 &0 &0 &2C_5 &2C_6\\\hline C_6& C_6&C_5 &0 & C_5&0&0 
\end{array} 
\end{equation}
	
It follows from Table \ref{Tab:tabla1} that the real associative subalgebra  of $\mathfrak{a}(G,\nabla^+)$ generated by $\{e_1^-,e_2^-\}$ has linear basis $\beta=(e_1^-,e_2^-,C_3,C_4,C_5)$. That is, env${}_{\nabla^+}(\mathfrak{g})$ is the real 5-dimensional associative algebra with linear basis $\beta$ and the opposite product of Table \eqref{Tab:tabla1}. 
	
The  animation below shows the lines of  flow of each of the vector fields $e_1^-,e_2^-,C_3,C_4,C_5$ and $C_6$ with the initial condition $(1.5,-1)$. The fact that the flows of the vector fields $C_3,C_4,C_5$ and $C_6$ cross the boundary of the orbit of $(0,0)$ determined by the affine \'etale representation relative to $\nabla$, correspond to the fact that the vector fields are not complete in $G_0$. (To play the animation click on the image on the pdf version, only accessible from the online version).
\begin{center}
\animategraphics[controls,loop,height=5cm]{4}{imagen}{1}{46}
\end{center}
\end{example}

\begin{remark} The reader can verify that the simply connected Lie group of Lie algebra  $env_{\nabla^+}(\mathfrak{g})_-$ of the previous example is isomorphic to the group   $\mathbb{E}:=(\mathbb{R}^2\rtimes_{\theta_1} \mathbb{R})\rtimes_{\theta_2} G$, where $G=\mathsf{Aff}(\mathbb{R})_0$ is the connected component  of the unit of $\mathsf{Aff}(\mathbb{R})$ and the actions $\theta_1$ and $\theta_2$ are respectively given by
\[   \theta_1(t)(x,y)=(e^{2t}x,e^{2t}y)\quad\text{ and }\quad  \theta_2(x,y)(z_1,z_2,z_3)=\left(x^2z_1-xyz_2+y^2z_3,xz_2-2yz_3,z_3\right) \] 
The Lie group $\mathbb{E}$ is also isomorphic to the semidirect product $\mathcal{H}_3\rtimes_\rho \mathbb{R}^2$ of the additive group $\mathbb{R}^2$ acting on the 3-dimensional Heisenberg group where $\rho$ is given by 
\begin{equation} \label{Eq:actiononheisenberg} \rho(a,b)(x,y,z)=\left(e^{a}x,e^{2a+2b}y,e^a(e^{2b}-1)x+e^{a+2b}z\right).\end{equation}
	
\end{remark}

At this point we have some questions for which we do not have an answer.

\textbf{Questions.}\begin{enumerate}
\item[1.]  To determine a transformation group $E$ of diffeomorphisms of $P=L(M)$ so that $\mathsf{Lie}(E)=env_\nabla(\mathsf{aff}(M,\nabla))$ in the case  when $\dim(env_\nabla(\mathsf{aff}(M,\nabla)))<n^2+n$. \vskip3pt
\item[2.] Is it possible to realize an enveloping Lie group  of a flat affine Lie group $(G,\nabla^+)$ as a group of transformations of $L(G)$ (eventually as a subgroup of the Lie group $K$)? 
\end{enumerate}
\noindent
If $\mathfrak{a}_c(P)=\mathfrak{a}(P,\omega)$ the answer to the first question is positive. When $\dim(Env(\mathfrak{a}_c(P)))<n^2+n$, the method described on Theorem \ref{T:envelopeofaflataffinemanifold} could give an answer to these questions. 

%

The following example describes the  generic case of flat affine connections on $\mathsf{Aff}(\mathbb{R})$.

\begin{example} \label{Ex:genericcase}
	Consider the family of left symmetric products on $\mathfrak{g}=\mathsf{aff}(\mathbb{R})$  given by
	$$\begin{array}{c|c|c}
	\cdot&e_1&e_2\\ \hline
	e_1&\alpha e_1& e_2\\
	\hline 	e_2&0&0 	\end{array}$$
	where $\alpha\ne0$  and  denote by $\nabla_\alpha^+$ the corresponding left invariant flat affine connection on $\mathsf{Aff}(\mathbb{R})_0=:G$. A calculation shows that $\mathsf{aff}(G,\nabla_\alpha^+)$ is generated by the vector fields  
	\begin{align*}
	C_1=x\frac{\partial}{\partial x}+y\frac{\partial}{\partial y}, && C_2=\frac{\partial}{\partial y}, && C_3=y\frac{\partial}{\partial y},
	& &\mbox{and}&& C_4=\frac{1}{\alpha}(x^{\alpha}-1)\frac{\partial}{\partial y}.
	\end{align*}
	The product defined as in Equation \ref{Eq:productinducedbyaconnection} gives 
	\begin{align*}
	\begin{array}{c|c|c|c|c}
	&C_1&C_2&C_3&C_4\\ \hline
	C_1&\alpha C_1-(\alpha-1)C_3&0&C_3&\alpha C_4+C_2\\ \hline
	C_2&C_2&0&C_2&0\\ \hline
	C_3&C_3&0&C_3&0\\ \hline
	C_4&C_4&0&C_4&0
	\end{array}
	\end{align*}
	Hence $env_{\nabla_\alpha^+}(\mathsf{aff}(G,\nabla_\alpha^+))=\mathsf{aff}(G,\nabla_\alpha^+)$.  Now, for  $\alpha\ne1$,  the associative envelope of the left symmetric algebra $(\mathfrak{g},\cdot)$ is  3-dimensional with basis $(C_1,C_2,C_3)$.  Whereas, for $\alpha=1$, the associative envelope of $(\mathfrak{g},\cdot)$ is two dimensional with linear basis $(C_1,C_2)$. 
	
Furthermore, using  Theorem \ref{T:envelopeofaflataffinemanifold},  an enveloping Lie group of  $\mathsf{Aff}(G,\nabla_\alpha^+)$ is given by
	\[ G'=\left\{\left(\begin{matrix}1+\beta_1&\beta_2&0\\0&1+\beta_3&0\\\beta_1/\alpha&\beta_2/\alpha+\beta_4&1\end{matrix}\right)\bigg|\ \beta_1\ne-1\ \text{and}\ \beta_3\ne-1\  \right\}.\]  
\end{example}



We finish the section with a more general example.

\begin{example} \label{Ex:glnr}  Let  $G=GL_n(\mathbb{R})$ endowed with the flat affine bi-invariant connection $D$ determined by composition of linear endomorphisms. Given the local coordinates $\left[x_{ij}\right]$ with $i,j=1,\dots,n$, it is easy to check that  linear bases of left and right invariant vector fields are given by
\[ E_{rs}^+=\sum_{i=1}^n x_{ir}\frac{\partial}{\partial x_{is}}\quad \text{and}\quad E_{rs}^-=\sum_{i=1}^n x_{si}\frac{\partial}{\partial x_{ri}} ,\]
with $r,s=1,\dots,n$. The group $\mathsf{Aff}(G,D)$ is of dimension $2n^2-1$ (\cite{BM})  and the Lie bracket of its Lie algebra is the  bracket of vector fields on $G$.  Using the product determined by $D$ on  left and right invariant vector fields one gets 
\[ D_{E_{pq}^+}E_{rs}^-= x_{sp}\frac{\partial}{\partial x_{rq}},\quad\text{for all}\quad p,q,r,s=1,\dots,n. \]
It follows that an enveloping Lie group of $\mathsf{Aff}(G,D)$ is locally isomorphic to $GL_{n^2}(\mathbb{R})$.
\end{example}

The following remark describes an algorithm to compute the associative envelope of a finite dimensional left symmetric algebra.

\begin{remark} Given a finite dimensional real or complex left symmetric algebra $A$, do as follows. 

 Find, applying Lie's third theorem, the connected and simply connected Lie group $G(A)$ of Lie algebra $A_-$ of commutators of $A$.  This group is endowed with a flat affine left invariant connection $\nabla^+$ determined by the product on $A$. 
 
Compute the associative subalgebra $\mathcal{B}$ of $\mathfrak{a}(G(A),\nabla^+)$ generated by  the subspace of right invariant vector fields.

The associative envelope $env_{\nabla^+}(A)$ of $A$ is $\mathcal{B}^{op}$.
\end{remark}

To finish the section, let us pose the following interesting problem.\medskip 

\noindent 
\textbf{Open Problem.} To find an algebraic method to determine the associative envelope, in the sense of the Definition \ref{D:envelopingofleftsymmetricalgebra}, of a real or complex finite dimensional left symmetric algebra.
 
\section{Affine Transformation Groups endowed with a flat Affine  bi-invariant structure} \label{S:studyofaffinetransformationgroups}
To finish, let us present some general cases of flat affine manifolds  giving a positive answer to Question \ref{Q:Medinasquestion}.  

\begin{proposition}\label{P:completnessimpliesbiinvariant} If $(M,\nabla)$ is a complete flat affine manifold, then $\mathsf{Aff}(M,\nabla)$ has a flat affine bi-invariant connection determined by $\nabla$. 
\end{proposition}
\begin{proof} It follows from the fact that the completeness of $\nabla$ implies that $\mathfrak{a}(M,\nabla)=\mathsf{aff}(M,\nabla)$ (see \cite{KN} page 234).
\end{proof}

\begin{corollary} 
	Let $(M,\nabla)$ be a connected and simply connected  flat affine manifold. If $\dim \mathsf{Aff}_x(M,\nabla)\geq n^2$ for some $x\in M,$ then the group $\mathsf{Aff}(M,\nabla)$ is endowed with a flat affine bi-invariant structure determined by $\nabla$.
\end{corollary}
\begin{proof} The result follows by recalling that the only obstruction to the completeness of $\nabla$ is the existence of a point $x\in M$ so that $\dim \mathsf{Aff}_x(M)<n^2$ (see \cite{BM} and \cite{T}). 
\end{proof} 

\begin{corollary} Let $(G,\nabla^+)$ be a flat afffine Lie group,  then the group $\mathsf{Aff}(G,\nabla^+)$ admits a flat affine bi-invariant connection if any of the following conditions hold
	\begin{enumerate}
		\item $(G,\omega^+)$ is a unimodular  symplectic Lie group and $\nabla^+$ is the flat affine connection naturally determined by $\omega^+$.
		\item $G$ is unimodular and $\nabla^+$ is the Levi-Civita connection of a flat pseudo-Riemannian metric.
	\end{enumerate}
\end{corollary}
\begin{proof} Cases (1) and (2) imply that the connection is complete (see \cite{LM} and \cite{AM}, respectively).
\end{proof}

The following result is due to \cite{V}, we present a different proof .

\begin{proposition}
	Let $(M,\nabla)$ be a compact flat affine manifold,  then the Lie group $\mathsf{Aff}(M,\nabla)$ is endowed with a flat affine bi-invariant connection determined by $\nabla$. 
\end{proposition}
\begin{proof} 
As $M$ is a compact manifold every vector field on $M$	is complete, in particular every infinitesimal affine transformation of $(M,\nabla)$ is complete, hence $\mathsf{aff}(M,\nabla)=\mathfrak{a}(M,\nabla)$. The conclusion follows from Lemma \ref{L:associativity}. 
\end{proof}

\begin{corollary} Let $(G,\nabla^+)$ be a flat affine Lie group,  $D$ a discrete cocompact subgroup  of $G$ and $\overline{\nabla}$ the connection naturally induced by $\nabla^+$ on $M={}_{D\backslash }G$. Then the group $\mathsf{Aff}(M,\overline{\nabla})$ has a flat affine bi-invariant structure determined by $\nabla$.  
\end{corollary}



Notice that the converse of Proposition \ref{P:completnessimpliesbiinvariant} is not true. A counterexample is given by the Hopf torus that admits a flat affine connection determined by the usual connection on the punctured plane (see \cite{Nag} p 200, see also \cite{Ben}), and therefore non-complete, but its group of affine transformations is locally isomorphic to $GL_2(\mathbb{R})$ (see \cite{Nag}).

\begin{corollary} The group of isometries $\mathfrak{I}(M,g)$ of a compact flat Riemannian manifold $(M,g)$, admits a flat affine bi-invariant structure determined by the Levi-Civita connection. 
\end{corollary}
\begin{proof} The hypothesis imply that $\mathfrak{I}(M,g)_0=\mathsf{Aff}(M,\nabla)_0$ (see \cite{Y} and \cite{KN} page 244).  
\end{proof}

\begin{corollary} If $M$ is compact and $\nabla$ is the Levi-Civita connection of a flat Lorentzian metric, the group of $\mathsf{Aff}(M,\nabla)$ admits a flat affine bi-invariant connection.
\end{corollary}
\begin{proof}  Under these assumptions, the flat manifold is geodesically complete (see \cite{Car}).
\end{proof}

\noindent\textbf{Acknowledgment.} We are very grateful to Martin Bordemann and to the referee for their many  comments and suggestions that  allowed to improve the redaction of this work. We  also thank to the referee  for communicating non published Jack Vey's results.


\end{document}